\documentclass[11pt]{amsart}
\usepackage{amsmath,amssymb,amsfonts,amscd,a4}
\usepackage[cmtip, all]{xy}
\theoremstyle{plain}
\newtheorem{thm}{Theorem}[section]

\newtheorem{prop}[thm]{Proposition}
\theoremstyle{definition}

\newtheorem{definition}[thm]{Definition}

\numberwithin{equation}{section}
\newcommand{\ga}[2]{\begin{gather}\label{#1}#2 \end{gather}}

\newcommand{\sA}{{\mathcal A}}
\newcommand{\sB}{{\mathcal B}}
\newcommand{\sC}{{\mathcal C}}

\newcommand{\sI}{{\mathcal I}}

\newcommand{\sO}{{\mathcal O}}

\newcommand{\sS}{{\mathcal S}}
\newcommand{\sT}{{\mathcal T}}
\newcommand{\sU}{{\mathcal U}}
\newcommand{\sV}{{\mathcal V}}
\newcommand{\sW}{{\mathcal W}}

\newcommand{\F}{{\mathbb F}}

\newcommand{\Q}{{\mathbb Q}}

\newcommand{\Z}{{\mathbb Z}}

\title[Deligne's finiteness ]{A remark on Deligne's finiteness theorem}
\author{H\'el\`ene Esnault } 
\address{Freie Universit\"at Berlin, Arnimallee 3, 14195, Berlin,  Germany}
\email{esnault@math.fu-berlin.de}

\thanks{The  author is supported by  the Einstein program. }

\date{\today}

\begin{document}

\begin{abstract}  Over a connected geometrically unibranch scheme  $X$ of finite type over a finite field, we show,  with  purely geometric arguments, finiteness of the number of  irreducible  $\bar \Q_\ell$-lisse sheaves, with  bounded rank and bounded ramification in the sense of Drinfeld,  up to twist by a character of the finite field.  On $X$ smooth, with bounded ramification in the sense of \cite[Defn.3.6]{EK12}, this is Deligne's theorem~\cite[Thm.1.1]{EK12}, the proof of which uses the whole strength of \cite{Del12} and \cite{Dri12}.  We also generalise Deligne's theorem~\cite[Thm.1.1]{EK12} from $X$ smooth to $X$ normal,  using  Deligne's theorem (not reproving it) and  a few more geometric arguments.
\end{abstract}
\maketitle

\section{Introduction}\label{intro}
Let $X$ be a normal geometrically connected scheme of finite type defined over a finite field $\F_q$, let $\sV$ be an irreducible $\bar \Q_\ell$-lisse sheaf with finite determinant, where $\ell \neq p={\rm char}(\F_q)$.  In Weil II  \cite[conj.1.2.10]{Del80}, Deligne conjectured the following. 
\begin{enumerate}
\item[(i)] $\sV$ has weight $0$.
\item[(ii)] There is a number field $E(\sV)\subset \bar \Q_\ell$ such that for all closed points $x\in X$, the characteristic polynomial $f_{\sV}(x)( t)={\rm det}(1-F_xt, \sV)\in E(\sV)[t]$, where $F_x$ is the geometric Frobenius of $x$.
\item[(iii)] For any $\ell' \neq p$ and any embedding $\sigma: E(\sV)\hookrightarrow \bar \Q_{\ell'}$, for any closed point $x$ of $X$, all eigenvalues of $F_x$  are $\ell'$-adic units.
\item[(iv)] For any $\sigma$ as in (iii), there is an irreducible   $\bar \Q_{\ell'}$-lisse  sheaf $\sV_{\sigma}$, called the companion to $\sigma$,  such that $\sigma f_{\sV}(x)=f_{\sV_\sigma}(x)$.
\end{enumerate}
\medskip

Deligne's conjectures have been proved by Lafforgue \cite[Thm.VII.6]{Laf02} when $X$ is a smooth curve. 

\medskip

From this one deduces (i)  and (iii) in higher dimension as follows. For any closed point $x$, one finds a smooth curve $C$ mapping to $X$ such that $x\to X$ lifts to $x\to C$, such that the restriction of $\sV$ to $C$ remains irreducible. To this aim,  Lafforgue in \cite[Prop.VII.7]{Laf02} uses Bertini's theorem, unfortunately  in a wrong way.  It has been corrected in \cite[Thm.2.15]{Dri12} using the Hilbert  irreducibility theorem (see  also \cite[ App.B ]{EK12}) and in  \cite[1.5-1.9]{Del12}  using the  Bertini theorem.  One first shows a group theoretic lemma saying that there is a finite quotient of the monodromy group of $\sV$ such that if the monodromy group of the restriction of $\sV$ to a subscheme still has this finite group as a quotient, then  it has the same monodromy group as on $X$. The construction of a curve 'doing' this is then performed by Bertini or Hilbert's irreducibility. 

\medskip

  Deligne proved (ii) in 2007 (\cite[Thm.3.1]{Del12}). Drinfeld, using (ii), proved (iv) in 2011 (\cite[Thm.1.1]{Dri12}), assuming in addition $X$ to be smooth.

\medskip

Let $X$ be a  geometrically connected scheme of finite type over a characteristic $p>0$ field $k$,  $\beta: Y\to X$ be a proper dominant morphism.  One says that a $\bar \Q_\ell$-lisse sheaf  $\sV$ has ramification bounded by $\beta$ if $\beta^*( \sV)$ is tame (see \cite[Thm.2.5]{Dri12}). 
This means that for any smooth curve $C$ mapping to $Y$, the pullback of $\sV$ to $C$ is tame in the usual sense (\cite[Defn.2.1.2,p.30]{GM71}). 
If $X$ is geometrically unibranch,  $\sV$ is defined by a representation $\rho: \pi_1(X)\to GL(r, R)$ of the fundamental group 
$\pi_1(X)$  defined in \cite[V]{SGA1} (see \cite[Intro., Lem.7.4.10]{BS15}, we omitted the base point here), where $R\supset \Z_\ell$ is a finite extension of discrete valuation rings.  Then $\sV$ has ramification bounded by   any finite \'etale Galois cover $\beta$ such that $\pi_1(Y)\hookrightarrow \pi_1(X)$ is an open normal pro-$\ell$-subgroup, for example $\pi_1(Y)={\rm Ker} \big(\pi_1(X)\to GL(r, R)\to GL(r, R/2\ell)\big)$. We say such an  $\beta$  makes $\sV$ tame. 
If $Y$ is smooth, the tame fundamental group  $\pi^t_1(Y)$ is well defined   as a quotient of $\pi_1(Y)$ (see \cite[Thm.1.1]{KS10}) and the property that $\beta^*(\sV)$ is tame is  equivalent to 
the property that the underlying representation of $\pi_1(Y)$ factors through $\pi^t_1(Y)$.

\begin{definition} \label{defn:alpha}
  Given a natural number $r$ and  given $\beta$,  one defines $\sS(X, r, \beta)$ to be  the set of rank $r$ isomorphism classes of irreducible $\bar \Q_\ell$-lisse sheaves, with ramification bounded by $\beta$,   modulo twist by a character of the Galois group of $k$.
\end{definition}

\medskip

Let $X$ be a  normal  scheme of finite type over  a characteristic $p>0$ field $k$,  $X\hookrightarrow \bar X$ be a normal compactification, and $D$ be an effective Cartier divisor with support $\bar X\setminus X$. One says that a $\bar \Q_\ell$-lisse sheaf  $\sV$ has ramification bounded by $D$ if for any smooth curve $C$ mapping to $X$, 
with compactification $\bar C\to \bar X$, where $\bar C$ is smooth, 
the pullback $\sV_C$ of $\sV$ to $C$ has Swan conductor bounded above by $\bar C\times_{\bar X} D$ (\cite[Defn.3.6]{EK12}). 

\begin{definition} \label{defn:D}
 Given a natural number $r$ and  given $D$,   one defines $\sS(X, r, D)$ to be  the set of isomorphism classes of rank $r$ irreducible $\bar \Q_\ell$-sheaves, with ramification bounded by $D$,  modulo twist by a character of the Galois group of $k$.

\end{definition}

Recall that if $k$ is a finite field,  by class field theory \cite[Thm.~1.3.1]{Del80}, any class in $\sS(X,r, \alpha)$ (Definition~\ref{defn:alpha})  or 
$\sS(X,r, D)$ (Definition~\ref{defn:D})  contains a $\sV$ with finite order determinant. 

\medskip

If $\sV$ has ramification bounded by a finite \'etale $\beta$, then  it also has ramification bounded by $\Delta$, where $\Delta$ is the discriminant of $\beta$ times the rank of $\sV$ (\cite[Prop.3.9]{EK12}), that is 
\ga{1}{ \sS(X,r,\beta) \subset \sS(X,r, \Delta).}   Given an effective Cartier divisor $D$ with support $\bar X\setminus X$,  we do not know whether there is a $\beta: X'\to X$  finite \'etale such that $\sS(X,r, D) \subset \sS(X,r, \beta)$. See Section~\ref{ss:ram}.

\medskip
In 2011, Deligne proved the following finiteness theorem (\cite[Thm.1.1]{EK12}).
\begin{thm} \label{thm:deligne}
Let $X$ be  a geometrically connected smooth scheme of finite type defined over a finite field $k$, $X\hookrightarrow \bar X$ be a normal compactification, $D$ be a Cartier divisor with support $\bar X\setminus X$, $r$ be a natural number. Then  $\sS(X,r,D)$ is finite.

\end{thm}
Of course in dimension $1$, the theorem is a consequence of Lafforgue's main theorem \cite[Prop.VII.7]{Laf02}.
Deligne's proof uses the full strength of  Drinfeld's theorem on the existence of the companions, which thus forces $X$ to be smooth. This in turn uses the full strength of Deligne's theorem on the existence of the number field.  One notices that 
abstractly, the finiteness theorem together with the existence of the companions implies the existence of the number field (see \cite[ Cor.8.3]{EK12}).

\medskip

In this short note we prove the following
\begin{thm} \label{thm:main}
Let $X$ be a geometrically connected, geometrically unibranch   scheme of finite type defined over a finite field, $\alpha: X'\to X$  be a finite \'etale cover, $r$ be a natural number. Then $\sS(X,r,\alpha)$ is finite.

 \end{thm}
 Our proof does not use the existence of the companions, nor does it use the existence of the number field. The main ingredient is the existence, on a good alteration \cite[Thm.4.1]{dJ97},  of a single curve on which all irreducible lisse tame $\bar \Q_\ell$-sheaves remain irreducible (\cite[Thm.1.1 (a)]{EKin15}). 

\medskip

We also prove  in Theorem~\ref{thm:deligne2} that in Deligne's Theorem~\ref{thm:deligne},  it is enough to assume $X$ to be normal. The proof uses the functoriality property of the notion of ramification bounded by $D$. Unlike for Theorem~\ref{thm:main}, we don't have a direct proof and have to use Theorem~\ref{thm:deligne}.

\medskip

 {\it Acknowledgements:} It is a pleasure to thank Moritz Kerz and Lars Kindler for earlier discussions, while writing \cite{EK12} and \cite{EKin15},  on topics related to this note. We thank Pierre Deligne for his comments and the questions he asked after we completed the proof of Theorem~\ref{thm:main}. 
  We thank the two referees for their careful reading and  for their friendly comments which helped us to improve the exposition of this note.

\section{Proof of Theorem~\ref{thm:main}}

By \cite[Lem.9.8]{StacksProject}, for any non-trivial open subscheme  $U\hookrightarrow X$, the homomorphism $\pi_1(U)\to \pi_1(X)$ is surjective (based at any geometric point of $U$). As  $X$ is geometrically unibranch,  any representative  $\sV$ of a class in $\sS(X,r,\alpha)$ is the isomorphism class of a representation of $\pi_1(X)$ (see introduction). Thus we may assume that $X$ is quasi-projective. 

\medskip

We fix $\alpha: X'\to X$  as in the theorem, and $\alpha': Y\to X'$ an alteration such that $\alpha'$ is  proper, generically \'etale, $Y$ is smooth and has a compactification $Y\hookrightarrow \bar Y$ which is projective, smooth, and such that $\bar Y\setminus Y$ is a strict normal crossings divisor (\cite[Thm.4.1]{dJ97}).  We denote by $k_Y$ the field of constants of $Y$. Let $\beta: Y\xrightarrow{\alpha'} X'\xrightarrow{\alpha} X$ be the composite morphism, and $U\hookrightarrow X$ be a non-trivial open of $X$ over which $\beta$ is \'etale. We set $V=\beta^{-1}(U)$. 
Again by the surjectivity $\pi_1(U)\to \pi_1(X)$, $\sV|_U$ is irreducible as well.  In addition, the restriction to $U$ induces an injective map  $\sS(X, r, \alpha)\to \sS(U, r,  \alpha|_U)$ of sets.
As $\beta|_U$ is \'etale and  the coefficients $\bar \Q_\ell$ have characteristic $0$, $\beta^{*}(\sV|_U)$ is then a sum of irreducible lisse $\bar \Q_\ell$-sheaves $\beta^{*}(\sV|_U)=\oplus_i \sV_i^U$. 

\medskip

As $\alpha^*(\sV)$ is tame, so is $\beta^*(\sV)$.  That is $\sS(X,r,\alpha)\subset \sS(X, r, \beta)$. The composite surjection $\pi_1(V)\twoheadrightarrow \pi_1(Y)\twoheadrightarrow \pi^t_1(Y)$ factors through $\pi_1^t(V) \to \pi_1^t(Y)$, which is thus surjective. Thus  for every $i$,  $\sV_i^U$ is the restriction to $V$ of a  uniquely defined lisse irreducible tame $\bar \Q_\ell$-sheaf  $\sV_i$ on $Y$ and $\beta^{*}(\sV)=\oplus_i \sV_i$.  Let $s \ge 1$ be any natural number. The restriction to $V$ induces an injective map  $\sS(Y, s, {\rm Id})\to \sS(V, s, {\rm Id})$ of sets.

\begin{prop} \label{prop:fin}
For any natural number $s\ge 1$, the set $\sS(Y, s,  {\rm Id})$ is finite.

\end{prop}
\begin{proof}
Let $\bar C\hookrightarrow \bar Y$ be a smooth curve, complete intersection of ample divisors  of sufficiently high degree, transversal to $\bar Y\setminus Y$, containing the geometric base point at which we centered $\pi_1(Y)$.  By \cite[Thm.1.1 (a)]{EKin15}, the homomorphism $\pi_1^t(C)\to \pi_1^t(Y)$ is surjective, where $C=\bar C\cap Y$.  Thus the restriction to $C$ induces an injective map   $\sS(Y, s, {\rm Id})\to \sS(C, s,  {\rm Id})$ of sets, and $\sS(C, s, {\rm Id})$ is finite by Lafforgue's theorem (see introduction).

\end{proof}
For any natural number $s$ such that $1\le s\le r$, and every class in the finite set $\sS(Y, s, {\rm Id})$ (Proposition~\ref{prop:fin}), we fix an irreducible $\bar \Q_\ell$-lisse \'etale sheaf $\sW$ representing this class.  We denote by $\sT(Y, r, {\rm Id})$ the finite set of $\sW$ so chosen, which is in bijection with $\bigcup_{s=1}^r \sS(Y, r, {\rm id})$.  
 Further, for each $\sW \in \sT(Y, r, {\rm id})$, the $\bar \Q_\ell$-lisse sheaf $ (\beta|_V)_* (\sW|_V)$ on $U$ contains only finitely many irreducible $\bar \Q_\ell$-lisse subsheaves $\sU$.  We denote by $\sC(U, r)$
the finite collection of $\sU$ so defined. 

\medskip

For any $i$,  there is a character $\chi(\sV_i)$ of the Galois group of $k_Y$ and a $\sW (\sV_i) \in \sT(Y, r, {\rm id})$ such that $\sV_i=\sW (\sV_i)\otimes \chi(\sV_i)$. 

\medskip

We first finish the proof under the assumption $k=k_Y$. 
For any $i$,  the composite map
$$\sV|_{U} \hookrightarrow (\beta|_{V})_* (   \beta|_V  )^* (\sV|_{U}) \xrightarrow{\rm projection} (   \beta|_{V}  )_*( \sV_i^U)= (   \beta|_{V}  )_*( \sW( \sV_i)|_V) \otimes \chi(\sV_i)$$
is injective and has image $\sU\otimes \chi(\sV_i)$ for one of the $\sU$ in $\sC(U,r)$. 
Thus the image of $\sS(X,r,\beta)$ in $\sS(U, r, \beta|_V )$ lies in the subset of classes of elements in $\sC(U,r)$, thus is finite.  Thus $\sS(X,r, \alpha)\subset \sS(X,r,\beta)$ is finite as well. 
\medskip

We now treat the general case.
Let  $\beta: Y\xrightarrow{\gamma}  X\otimes k_Y  \xrightarrow{\epsilon} X$ be the factorization. 
We know by the case $k=k_Y$ that $\sS(X\otimes k_Y, s, \gamma)$ is finite for all $1\le s \le r$. The only remaining issue is to compare  the different notions of twist by a character. 

\medskip
For an irreducible $\bar \Q_\ell$-lisse sheaf $\sV$ of rank $r$ on $X$, 
 $\epsilon^*(\sV)=\oplus_{i=1}^N \sA_i$, where $\sA_i$ is  an irreducible $\bar \Q_\ell$-lisse sheaf $\sA_i$, and where the Galois group $\Z/m$  of $k_Y \cong \F_{q^m}$ over $k\cong \F_q$ 
acts  on the set $\{\sA_i, i=1,\ldots, N\}$, transitively as $\sV$ is irreducible, and  via its quotient $\Z/N$. 
Let $\sigma$ be the image of Frobenius of $k$ in  $\Z/N$, then $\epsilon^*(\sV)=\sA \oplus \sigma^*\sA \oplus  \ldots \oplus (\sigma^{N-1})^*\sA$, where $\sA$ can be chosen to be any of the $\sA_i$, and $\sV =\epsilon_* ((\sigma^i)^* \sA)$ for any $i=0, \ldots, N-1$.  In particular, $r=a\cdot N$ and $N$ divides $m$, where $a$ is the rank of $\sA$.  As the Galois group of $k$ is abelian,  the Galois group of $k_Y$ over $k$ acts trivially on the Galois group of $k_Y$, thus if $\chi$ is a character of the Galois group of $k$, $\sigma^* \epsilon^*(\chi) =\epsilon^*(\chi)$. We conclude that $\epsilon^*(\sV \otimes \chi)= \sA  \otimes  \epsilon^*(\chi) \oplus \sigma^*(  \sA  \otimes \epsilon^*(\chi) ) \oplus \ldots \oplus (\sigma^{N-1})^* (\sA \otimes  \epsilon^*(\chi)) .    $

\medskip
We now compare   $\sS(X,r, \beta)$ and $\sS(X\otimes k_Y, s, \gamma)$ for $1\le s\le r$.
The $\bar \Q_\ell$-lisse sheaf $\epsilon^*(\sV \otimes \chi)$, for a character $\chi$ of the Galois group of $k$, is uniquely determined by $\sA \otimes \epsilon^*(\chi)$.  A character of the Galois group of $k_Y$ is determined by the value of the Frobenius of $k_Y$,  which is an $\ell$-adic unit $u \in \bar \Z_\ell ^\times \subset \bar \Q_\ell^\times$, thus is a $m$-th power, thus  comes from a character of the Galois group of $k$.  
Thus $\sS(X\otimes k_Y, a, \gamma)$ is equal to  the set  of irreducible $\bar \Q_\ell$-lisse sheaves $\sA$ of rank $a$ on $X\otimes k_Y$, such that $\gamma^* \sA$ is tame, modulo twists by the pull-back via $\epsilon$ of  a  character of  the Galois group of $k$.
  Finally, if the classes of $\sV$ and $\sV'$ in $\sS(X,r,\beta)$ are such that, writing $\epsilon^*(\sV)=\oplus (\sigma^i)^*\sA, \  \epsilon^*(\sV')=\oplus (\sigma^i)^*(\sA')
$ as above, the classes of $\sA$ and of $\sA'$ in $\sS(X\otimes k_Y, a, \gamma)$ are the same, then there is a character $\chi$ of the Galois group of $k$ such that $\sA=\sA'\otimes \epsilon^*( \chi)$. This implies that   $\sV=\epsilon_* (\sA) = \epsilon_* (\sA')\otimes \chi =\sV'\otimes \chi$.  Thus the classes of $\sV$ and $\sV'$ in 
$\sS(X,r, \beta)$ are the same. As $1\le a\le r$, thus $a$ is bounded, we conclude that  $ \sS(X,r,\beta)$ is finite, and so is $\sS(X, r, \alpha)\subset \sS(X,r,\beta)$. This finishes the proof.

\section{Remarks and Comments} \label{sec:comments}
\subsection{Skeleton sheaves}
In order to prove Theorem~\ref{thm:deligne}, Deligne introduces the notion of what we called $\bar \Q_{\ell}$-lisse $2$-skeleton sheaves in \cite[2.2]{EK12}. Those are collections $\{\sV_C\}_C$ of $\bar \Q_\ell$-\'etale sheaves $\sV_C$ for any smooth curve $C$ mapping to $X$, together with gluing conditions. On $C\times_XC'$, with projections $p_C, p_{C'}$ to $C$ and $C'$, one has an isomorphism $p_C^*V_C\cong p_{C'}^*V_{C'}$ which satisfies the cocyle condition (the definition is expressed slightly differently in {\it loc.cit.} and is trivially equivalent to this one).  An irreducible $2$-skeleton sheaf is one which does not contain any non-trivial sub ({\it loc.cit.}).  With the assumptions as in Definition~\ref{defn:alpha}, for $\alpha: X'\to X$ finite \'etale, one defines $\sS k(X,r,\alpha)$ to be the set of  rank $r$ irreducible $\bar \Q_\ell$-lisse  $2$-skeleton sheaves, with ramification  bounded by $\alpha$,  modulo twist by a character of the Galois group of $k$.  The boundedness by $\alpha$ means that for any smooth curve $C$ mapping to $X'$, $\sV_C$ is tame. 
 With the assumptions as in Definition~\ref{defn:D}, one defines $\sS k(X,r, D)$ to be the set of
 rank $r$  irreducible $\bar \Q_\ell$-lisse $2$-skeleton sheaves,  with  ramification  bounded by $D$, modulo twist by a character of the Galois group of $k$.   The boundedness by $D$ means that for any smooth curve $C$ mapping to $X$,  with normal compactification $\bar C$,  $\sV_C$  has ramification bounded by $\bar C\times_{\bar X} D$. Pull-back to curves induces maps of sets $\sS(X,r, \alpha) \to \sS k(X,r,\alpha)$ and $\sS(X,r, D)\to \sS k(X,r,D)$,  which are injective ({\it loc. cit.}).  Furthermore, the proof of \cite[Prop.~3.9]{EK12} shows that for $\Delta$ as in \eqref{1}, one has $\sS k(X,r, \alpha)\subset \sS k(X,r, \Delta). $ The proof of Theorem~\ref{thm:main} works word by word to show that under the assumptions of the theorem, $\sS k(X,r, \alpha)$ is finite.

\subsection{Two ways of bounding the ramification} \label{ss:ram}
As already mentioned,  under the assumptions of \eqref{1},
given an effective Cartier divisor $D$ with support $\bar X\setminus X$,  with $X$ smooth of finite type over a characteristic $p>0$ field $k$, we do not know whether there is a finite \'etale  morphism $\alpha: X'\to X$  such that $\sS(X,r, D) \subset \sS(X,r, \alpha)$. 

\medskip

If $k$ is finite, then, using Deligne's Theorem \ref{thm:deligne}, the answer is positive. One takes an $\alpha$ which makes   the direct sum representation 
$\oplus_{[\sV]\in \sS(X,r,D)  } \sV$ tame, where for each class $[\sV]$ in $\sS(X,r,D)$, $\sV$ is the choice of a representative. 

\medskip 
However, the question can be posed geometrically, that is assuming $k$ to be algebraically closed. 
In general, Deligne \cite{Del16} raises the following question. Given $X$ smooth of finite type over an algebraically closed field $k$, with a normal compactification $X\hookrightarrow \bar X$, given  an effective Cartier divisor $D$ with support $\bar X\setminus X$,  and a natural number $r$, does there exist a smooth curve $C\hookrightarrow X$ such that if 
$\sV$ is a representative of a class in $\sS(X,r,D)$, 
 its restriction $\sV_C$ to $C$ is still irreducible?  One can ask an even more optimistic question, dropping the boundedness of the rank $r$.
 Can one expect  the Lefschetz theorem \cite[Thm.1.1 (a)]{EKin15} to be true after replacing the tame fundamental group $\pi_1^t(X)$  by the Tannaka group of all $\bar \Q_{\ell}$-lisse   \'etale sheaves with ramification bounded by $D$?  

\medskip

For $X\hookrightarrow \bar X$ a good compactification,  there is a positive answer for the abelian quotient of the fundamental group in characteristic $\ge 3$  (see \cite[Thm.~1.1]{KS15}). 

\medskip

Deligne raises also in {\it loc. cit.} a sub-question. Consider all the finite Galois  \'etale covers $\alpha: X'\to X$  which make rank $r$  irreducible $\bar \Q_\ell$-lisse sheaf $\sV$ on $X$ with ramification bounded by $D$ tame.  For each of them, let  $\bar X'\to \bar X$ be the normalisation of $\bar X$ in the field of functions of $X'$.  Can one bound the inseparable  degree of the induced finite covers of the components of $\bar X\setminus X$?

\subsection{Smooth versus normal.}
The set $\sS(X,r,D)$ is defined under the assumption that $X$ is normal, in particular geometrically unibranch.  The spirit of the proof of Theorem~\ref{thm:main} enables us to generalize Deligne's finiteness theorem~\ref{thm:deligne}, from $X$ smooth to $X$ normal. Unfortunately, contrary to  Theorem~\ref{thm:main},  one has to use Theorem~\ref{thm:deligne}, so it does not  really shed a new light on it.

\begin{thm} \label{thm:deligne2}

Let $X$ be  a geometrically connected normal scheme of finite type defined over a finite field, $X\hookrightarrow \bar X$ be a normal compactification, $D$ be a Cartier divisor with support $\bar X\setminus X$, $r$ be a natural number. Then  $\sS(X,r,D)$ is finite.

\end{thm}

\begin{proof}
Let $j: U\hookrightarrow X$ be the smooth locus, 
let $\sI$ be the ideal sheaf defining the complement in $\bar X$ of the open embedding $U\hookrightarrow X \hookrightarrow \bar X$.   Let $\bar f:\bar Y\to \bar X$ be the normalization of the blow up of $\sI$, with restriction $f: Y=\bar f^{-1}(X)\to X$ to $X$.  The open embedding $j$ lifts to  the open embedding $j': U\to Y$.  
 Let $\sO_{\bar Y}(-E)=\bar f^*\sI$ be the locally free ideal sheaf  of the exceptional locus, and define the effective Cartier divisor $D'=E+\bar f^*D$ on $\bar Y$. It has support $\bar Y\setminus U$. 
\medskip

As $X$ is normal, thus  geometrically unibranch,  a representative $\sV$ of a class in $\sS(X,r,D)$  corresponds to a continuous representation $\rho: \pi_1(X)\to GL(r, \bar \Q_\ell)$ (\cite[Intro., Lem.7.4.10]{BS15}). If $h: C\to Y$ is a morphism from a smooth curve, with unique extension $\bar h: \bar C\to \bar Y$, where $\bar C$ is the normal compactification,  $(f\circ h)^*\sV$ has ramification bounded by $\bar h^* \bar f^*D$, thus,  a fortiori, $\sV|_U=f^*\sV|_U$ has ramification bounded by $D'$. 
As $\pi_1(U)\to \pi_1(X)$ is surjective \cite[Lem.9.8]{StacksProject}, restriction to $U$ induces an injective map $\sS(X,r,D)\to \sS(U, r, D')$. As  $\sS(U, r, D')$ is finite by Theorem~\ref{thm:deligne},  so is $\sS(X, r, D)$.

\end{proof}

\end{document}